\documentclass[10pt]{article} 

\usepackage{amsmath}
\usepackage{amssymb}
\usepackage{amsthm}
\usepackage[latin1]{inputenc}
\usepackage{graphicx}

\usepackage{ifthen}
\usepackage{color}

\newcommand{\intav}[1]{\mathchoice {\mathop{\vrule width 6pt height 3 pt depth  -2.5pt
\kern -8pt \intop}\nolimits_{\kern -6pt#1}} {\mathop{\vrule width
5pt height 3  pt depth -2.6pt \kern -6pt \intop}\nolimits_{#1}}
{\mathop{\vrule width 5pt height 3 pt depth -2.6pt \kern -6pt
\intop}\nolimits_{#1}} {\mathop{\vrule width 5pt height 3 pt depth
-2.6pt \kern -6pt \intop}\nolimits_{#1}}}

\def\polhk#1{\setbox0=\hbox{#1}{\ooalign{\hidewidth\lower1.5ex\hbox{`}\hidewidth\crcr\unhbox0}}}

\DeclareMathOperator{\Tr}{Tr}

\newtheorem{teo}{Theorem}

\newtheorem{Definition}{Definition}
\newtheorem{Lemma}{Lemma}

\newtheorem{Proposition}{Proposition}
\newtheorem{Remark}{Remark}
\newtheorem{Assumption}{A}

\def \suchthat {\ \big | \ }

\begin{document}

\title{\bf Fractional Sobolev regularity for fully nonlinear elliptic equations}
\author{Edgard A. Pimentel, Makson S. Santos, and Eduardo V. Teixeira}

\date{\today} 

\maketitle

\begin{abstract}

\noindent  We prove higher-order fractional Sobolev regularity for fully nonlinear, uniformly elliptic equations in the presence of unbounded source terms. More precisely, we show the existence of a universal number $0< \varepsilon <1$, depending only on ellipticity constants and dimension, such that if $u$ is a viscosity solution of $F(D^2u) = f(x) \in L^p$, then $u\in  W^{1+\varepsilon,p}$, with appropriate estimates. Our strategy suggests a sort of fractional feature of fully nonlinear diffusion processes, as what we actually show is that $F(D^2u) \in L^p \implies (-\Delta)^\theta u \in L^p$, for a universal constant $\frac{1}{2} < \theta <1$.  We believe our techniques are flexible and can be adapted to various models and contexts.

\medskip


\medskip 

\end{abstract}

\section{Introduction}\label{introduction}

We establish higher-order fractional regularity for the viscosity solutions of uniformly elliptic equations of the form
\begin{equation}\label{eq_main}
	F(D^2u) = f(x) \hspace{.4in}\mbox{in}\hspace{.1in} B_1,
\end{equation}
where $F\colon \mathcal{S}(d)\to\mathbb{R}$ is a $(\lambda,\Lambda)$-elliptic operator, and the source term $f\in L^p(B_1)\cap C(B_1)$, for $d<p\leq\infty$. We prove new interior estimates in Sobolev spaces $W^{1+\varepsilon,p}_{\rm loc}(B_1)$, for a universal constant $0< \varepsilon < 1$. 

The regularity theory for fully nonlinear elliptic equations has played a prominent role in mathematical analysis since its launch in the early 1980's and, while substantial advances have been made by several authors through the years, fundamental questions remain open until today. The first development in the theory is the Krylov-Safonov Harnack inequality \cite{krysaf}, which can be used to produced universal estimates in the $C^\alpha$ spaces, for a universal parameter $\alpha\in(0,1)$, for {solutions} of 
\[
	F(D^2u)=0\hspace{.2in}\mbox{in}\hspace{.2in} B_1.
\]
Through a linearization argument, $C^{1,\alpha}$-regularity results can be obtained; see, for instance, \cite{caf, trudinger}. Under convexity assumption upon the operator, the regularity of the solutions switches to the classical regime. That is, solutions to $F=0$ become $C^{2,\alpha}$, with estimates; this is the subject of the so-called Evans-Krylov theory \cite{krylovc2a}. 

In \cite{caf}, Caffarelli examines non-homogeneous fully nonlinear elliptic equations in the presence of variable coefficients. The findings reported in that paper cover regularity results in H\"older spaces -- $C^{1,\alpha}$ and $C^{2,\alpha}$ estimates. Caffarelli also launches in that paper a regularity theory in Sobolev spaces, by establishing estimates in $W^{2,p}$, for $p>d$. An essential assumption in the $W^{2,p}$-regularity theory is a convexity-like condition on the operator. To be precise, in \cite{caf}, it is assumed the homogeneous equation driven by the fixed-coefficients counterpart of the operator to have $C^{1,1}$-estimates. 
The requirement $p>d$ is weakened in \cite{escauriaza}, by means of an improved Harnack inequality proven in \cite{fabesstroock}. Caffarelli's estimates are then established in $W^{2,p}$ for $d/2\leq p_0<p$; the universal constant $p_0=p_0(\lambda,\Lambda,d)$ is called \emph{Escauriaza exponent} and plays a fundamental role in the theory of viscosity solutions of fully nonlinear equation. See \cite{M5} for sharp regularity estimates in such a regime. Concerning regularity estimates in H\"older spaces, if $p>d$, solutions are $C^{1,\mu}$-regular, with estimates; see e.g. \cite{caf, cafcab, M5, S}.

Furthering the regularity theory in Sobolev spaces, local estimates in $W^{1,q}$ are the subject of \cite{S}. In that paper, the author considers viscosity solutions to 
\[
	F(D^2u,Du,u,x)=f\hspace{.2in}\mbox{in}\hspace{.2in}B_1
\]
under usual structure conditions on $F$; see \cite{CCKS}. By supposing $f\in L^p(B_1)$, with $p_0<p\le d$, the author establishes estimates in $W^{1,q}_{{\rm loc}}(B_1)$ for every $q>1$ satisfying
\[
	q<p^*:=\frac{dp}{d-p}\hspace{.2in}\mbox{and}\hspace{.2in}d^*=\infty.
\]
While striking, such estimates are restricted to the level of the gradient of the solutions. This raises a fundamental question: what's the \emph{optimal}  degree of differentiability of solutions under merely uniform ellipticity? 

As concerns optimal regularity in H\"older spaces, this question was settled by the program carried out in \cite{NV1,NV2,NV3,NV4}. In those papers, the authors show that $C^{1,\alpha}$ estimates are indeed optimal. In Sobolev spaces, however, optimality remained largely open.

Our main result amounts to an integral estimate for fractional derivatives of order $1+\varepsilon$, where $\varepsilon\in(0,1)$ is a universal number, i.e. depends only on dimension and ellipticity constants (see the statement of Theorem \ref{theo_main} below). As a conclusion, one finds that ellipticity enforces a universal control on weak derivatives of order \emph{strictly higher than 1}. Our main result reads as follows:

\begin{teo}[Fractional Sobolev Regularity]\label{theo_main}
Let $u \in   C (B_1)$ be a viscosity solution to \eqref{eq_main}. Assume $F$ is $(\lambda, \Lambda)$-elliptic and $d<p < \infty$. Let $\varepsilon\in(0,\alpha_0)$, where $\alpha_0\in(0,1)$ is the exponent associated with the H\"older regularity for $F=0$. Then $u\in W_{\rm loc}^{(1+\varepsilon), p}(B_{1})$. In addition, there exists a positive constant $C=C(d,\lambda,\Lambda, \varepsilon,p)$ such that
\begin{equation}\label{est main thm}
	\|u\|_{W^{(1+\varepsilon), p}({B}_{1/2})} \leq C\left(\left\|u\right\|_{L^\infty(B_1)}+\left\|f\right\|_{L^p(B_1)}\right).
\end{equation}
\end{teo}

We stress that Theorem \ref{theo_main} is new for dimensions $d\geq 3$, since convexity of the operator $F$ is not required for the existence of $C^2$-regular solutions if $d=2$.

The proof of Theorem \ref{theo_main} draws inspiration from the arguments in \cite{caf}, leading to $W^{2,p}$-estimates for convex operators. However, because we only impose uniform ellipticity, $F$-harmonic functions are not entitled to second-order estimates, and touching the graph of the solutions with paraboloids is not effective in the present context. Hence an alternative must be designed to accelerate key decay rate appearing in the argument. The first main key novelty of our approach is an alternative geometric construction that considers $C^{1,\alpha}$-cones of the form
\[
	\varphi(x):=\ell(x)\pm M|x-x_0|^{1+\alpha},
\]
where $\ell \colon B_1\to\mathbb{R}$ is an affine function and $M>0$ is the opening of the cone. We are interested in the sets of points at which one can touch the graph of the solutions with cones of a certain opening. The idea is to show that information upon the measure of such sets can be translated into integral estimates of fractional order. 

Indeed, an interesting corollary of our findings sheds light on an important issue underlying the regularizing effects in the fully nonlinear setting. Our reasoning suggests, at least heuristically, that while of second order, the diffusion process associated with $F$ is no more efficient than $\alpha$-stable L\'evy process. See \cite{M1}; see also \cite{TeixNotices} for further, heuristic insights on diffusion efficiency versus regularity theory.

The remainder of this paper is organized as follows: Section \ref{sec_assump} presents our main assumptions, whereas Section \ref{sec_nr} gathers a few elementary notions and previous results. In Section \ref{sec_tools}, we resort to some geometric-measure techniques to produce a preliminary level of integrability for the aperture function. Section \ref{sec_prov} refines this integrability using geometric and approximation methods. Sections \ref{sec_tools} and \ref{sec_prov} follow the reasonings of \cite[Chapter 7]{cafcab} closely. The proof of Theorem \ref{theo_main} is the subject of Section \ref{sec_proofs}, where we show viscosity solutions of \eqref{eq_main} also verify a fractional diffusion equation.

\section{Preliminaries}\label{sec_mapm}

This section details the main hypotheses under which we work and presents preliminary facts and results used throughout the paper.

\subsection{Main assumptions}\label{sec_assump}

In what follows, we put forward the main assumptions used in the article. We start with ellipticity condition on the operator $F$.

\begin{Assumption}[Uniform ellipticity]\label{assump_ellip} 
The operator $F \colon S(d) \to \mathbb{R}$ is $(\lambda,\Lambda)$-uniformly elliptic. That is, for every $M, N\in\mathcal{S}(d)$ with $N\geq 0$ we have
\[
	\lambda\|N\|\leq F(M)-F(M+N)\leq\Lambda\|N\|.
\]
\end{Assumption}

Ellipticity can also be phrased in terms of the extremal Pucci operators. For $M\in \mathcal{S}(d)$, we define the Pucci extremal operators $\mathcal{P}_{\lambda,\Lambda}^\pm$ as
\[
{\cal P}_{\lambda,\Lambda}^+(M) := \sup_{A\in {\cal A}_{\lambda,\Lambda}}(-\Tr(AM))
\]
and 
\[
{\cal P}_{\lambda,\Lambda}^-(M) := \inf_{A\in {\cal A}_{\lambda,\Lambda}}(-\Tr(AM)),
\]
where 
\[
	{\cal A}_{\lambda,\Lambda} := \{A \in \mathcal{S}(d) : \lambda I\leq A \leq \Lambda I\}
\]
is the class of $(\lambda,\Lambda)$-elliptic matrices. It is important to note that ${\cal P}_{\lambda,\Lambda}^+(M) = -{\cal P}_{\lambda,\Lambda}^-(-M)$. The condition in A\ref{assump_ellip} is equivalent to
\[
	{\cal P}_{\lambda,\Lambda}^-(N) \leq F(M+N) - F(M) \leq {\cal P}_{\lambda,\Lambda}^+(N)
\]
for any $M,N \in S(d)$. Our next assumption concerns the integrability condition imposed on the source term $f$.

\begin{Assumption}[Integrability of the source term $f$]\label{assump_source}
Fix $p>d$. We suppose $f \in L^{p}(B_1)\cap C(B_1)$. In addition, there exists a constant $C>0$ such that
\[
	\left\|f\right\|_{L^{p}(B_1)} \leq C.
\]
\end{Assumption}

\begin{Remark}[Smallness regime]
Our arguments depend on a smallness regime on the $L^p$-norm of the source term $f$; i.e., we require
\[
	\left\|f\right\|_{L^p(B_1)}\leq \epsilon,
\]
for some $0<\epsilon \ll1$, to be universally determined further. From usual scaling arguments, it shall be clear that such smallness requirement neither imposes further constraints on the data of the problem nor affects the resulting estimates.
\end{Remark}

\begin{Remark}
Throughout the paper, we suppose that $F$-harmonic functions are in $W^{2,p}_{loc}(B_1)$ and make no use of estimates in such spaces. Such a condition is not restrictive, in light of an approximation strategy as the one put forward in \cite[Section 8]{PT}, see also \cite[Theorem 1.1]{PT}.
\end{Remark}

\subsection{Notations and preliminary results}\label{sec_nr}

For $r >0$ and $x_0 \in \mathbb{R}^d$, $B_r(x_0)$ denotes the open ball of radius $r$ centered at $x_0$. For simplicity, we denote $B_r(0)$ with $B_r$. Similarly, $Q_r(x_0)$ will denote the open cube with side $r$ and center $x_0$, i.e.,
\[
Q_r(x_0) := \left\lbrace x\in \mathbb{R}^d: |x-x_0|_\infty <\frac{r}{2}\right\rbrace,
\]
where $|x|_\infty := \max\{|x_1|, \ldots|x_d|\}$. The space of real symmetric $d\times d$ matrices is denoted with $S(d)$. Lastly, we mention that some constants appearing in the paper depend only on the dimension $d$ and the ellipticity $\lambda$ and $\Lambda$; we refer to such a constant as \emph{universal}. Lastly, we call a function $v:B_1\to\mathbb{R}$  normalized if $\|v\|_{L^\infty(B_1)}\leq 1$.

For completeness, we next include the definition of viscosity solution.

\begin{Definition}[Viscosity solution]\label{def_viscosity}
We say that $u\in C(B_1)$ is a viscosity sub-solution to
\begin{equation}\label{eq_defvisc}
	F(D^2u,Du,u,x) = 0\hspace{.4in}\mbox{in}\hspace{.1in}B_1
\end{equation}
if for every $x_0\in B_1$ and every $\phi\in W^{2,p}_{{\rm loc}}(B_1)$ such that $u-\phi$ attains a local maximum at $x_0$, we have
\[
	F(D^2\phi(x),D\phi(x),u(x),x)\leq 0.
\]
We say that $u\in C(B_1)$ is a viscosity super-solution to \eqref{eq_defvisc} if for every $x_0\in B_1$ and every $\phi\in W^{2,p}_{{\rm loc}}(B_1)$ such that $u-\phi$ attains a local minimum at $x_0$, we have
\[
	F(D^2\phi(x),D\phi(x),u(x),x)\geq 0.
\]
If $u\in C (B_1)$ is a viscosity sub and a super-solution to \eqref{eq_defvisc}, we say that $u$ is a viscosity solution to the equation.
\end{Definition}

For $g\in L^1_{\rm loc}(\mathbb{R}^d)$, we define the maximal function of $g$, denoted with $m(g)$, as
\[
m(g)(x) := \sup_{r>0}\frac{1}{|Q_r(x)|}\int_{Q_r(x)}|g(y)|{\rm d}y.
\]
We recall that the maximal operator satisfies 
\begin{equation}\label{max_op}
|\{x\in\mathbb{R}^d : m(g)(x)\geq t\}| \leq \frac{C}{t}\|g\|_{L^1(\mathbb{R}^d)}, \;\;\forall \;t>0.
\end{equation} 

Two important structures in our analysis are the convex envelope of a function and the associated contact set.
\begin{Definition}
Let $\Omega \subset \mathbb{R}^d$ be an open set and $v \in   C (\Omega)$. The convex envelope of $v$ in $\Omega$ is defined by
\[
\Gamma(v)(x) := \sup_{L}\{L(x)\,|\,L \leq v  \;\;\text{in}\;\;  \Omega, \; L \; \text{is affine}\}.
\]
The contact set of v is given by
\[
\{x \in \Omega\;|\; v(x) = \Gamma(v)(x)\}.
\]
\end{Definition}

Because our results rely solely on the ellipticity of the operator $F$, our strategy focuses on the $C^{1,\alpha}$-geometry of solutions. Consequently, we are interested in a contact set related to functions of class $C^{1,\alpha}$ for specific values of $\alpha\in(0,1)$. We proceed with the definition of $ C ^{1,\alpha}$-cone for $\alpha\in(0,1]$.
\begin{Definition}[$ C ^{1,\alpha}$-cone of opening $M$ and vertex $x_0$]
We say that $\psi$ is a convex $ C ^{1,\alpha}$-cone of opening M and vertex $x_0$ if
\[
	\psi(x) = L(x) + \frac{M}{2}|x - x_0|^{1 + \alpha},
\]
where $M$ is a positive constant, and $L(x)$ is an affine function. Similarly, $\psi$ is a concave  $ C ^{1,\alpha}$-cone of opening M and vertex $x_0$ if
\[
	\psi(x) = L(x) - \frac{M}{2}|x - x_0|^{1 + \alpha},
\]
where $M$ is a positive constant, and $L(x)$ is an affine function.
\end{Definition}

The sets collecting the points that can be touched by a $ C ^{1,\alpha}$-cone of certain opening $M$ play a pivotal role in our analysis. In fact, we produce a decay rate for their measure in terms of $M$ and relate this information with a distribution function appearing later (see Definition \ref{def_aperture}). Our next definition rigorously introduces those sets.

\begin{Definition}
Let $O \subset \Omega$ be an open subset, $0<\tau_0<\frac{{\rm diam}(O)}{5}$, and $M > 0$. We define
\[
	\underline{G}_M(u, O) = \underline{G}_M(O)
\]
as the set of all points $z \in O$ for which there exists a concave $ C ^{1,\alpha}$-cone $\psi$ of opening $M$ such that 
\begin{enumerate}
	\item $u(z) = \psi(z)$;
	\item $u(x) > \psi(x)$ for all $x \in B_{\tau_0}(z)$.
\end{enumerate}
We also define:
\[
	\overline{G}_M(u, O) = \overline{G}_M(O)
\]
as the set of all points $z \in O$ for which there exists a convex $ C ^{1,\alpha}$-cone $\psi$ of opening $M$ such that
\begin{enumerate}
	\item $u(z) = \psi(z)$;
	\item $u(x) < \psi(x)$ for all $x \in B_{\tau_0}(z)$.
\end{enumerate}

Finally 
\[
G_M(O) = \underline{G}_M(O)\cap \overline{G}_M(O).
\]
\end{Definition}

Next, we note a monotonicity property related to the sets $G_M$. Let $M_1 > M_2$ and take $x_1 \in \underline{G}_{M_2}(u,O)$. By definition, there exists a concave $ C ^{1,\alpha}$-cone of the form
\[
	\psi(x) = L(x) - \frac{M_2}{2}|x - x_1|^{1+\alpha},
\] 
where $L$ is an affine function, such that $\psi(x_1) = u(x_1)$ and $\psi(x) < u(x)$, for all $x \in B_{\frac{{\rm diam}(O)}{10}}(x_1)$. Hence,
\[
	u(x) > L(x) - \frac{M_2}{2}|x - x_1|^{1+\alpha} > L(x) - \frac{M_1}{2}|x - x_1|^{1+\alpha}= \tilde{\psi}(x),
\]
for all $x \in B_{\frac{{\rm diam}(O)}{10}}(x_1)$. Notice that, 
\[
u(x_1) = \psi(x_1) = \tilde{\psi}(x_1).
\] 
We conclude that $x_1 \in \underline{G}_{M_1}(u,O)$; thus
\[
\underline{G}_{M_2}(u,O) \subset \underline{G}_{M_1}(u,O).
\]
Similarly, we have
\[
\overline{G}_{M_2}(u,O) \subset \overline{G}_{M_1}(u,O).
\]
Therefore,
\[
G_{M_2}(u,O) \subset G_{M_1}(u,O).
\]

The following definition accounts for a family of sets collecting points where a $ C ^{1,\alpha}$-cone cannot touch the graph of the solutions.

\begin{Definition}
Let $O \subset \Omega$ be an open subset, $0<\tau_0<\frac{{\rm diam}(O)}{5}$ and $M > 0$. We define
\[
\underline{A}_M(u,O) = \underline{A}_M(O) = O\backslash \underline{G}_M(u,O). 
\]
Similarly 
\[
\overline{A}_M(u,O) = \overline{A}_M(O) = O\backslash \overline{G}_M(u,O).
\]
Finally
\[
A_M(u,O) = A_M(O) = O\backslash G_M(u,O).
\]
\end{Definition}

Now, we define the $ C ^{1,\alpha}$-aperture function. This structure is directly related to the integrability of the solutions to \eqref{eq_main}.

\begin{Definition}[$ C ^{1,\alpha}$-aperture function]\label{def_aperture}
Let $B\Subset B_1$. For $x \in B$ we define
\[
	\theta(x) := \theta_{1+\alpha}(u,B)(x) = \inf \left\lbrace M \suchthat x\in G_M(u,B) \right\rbrace\in [0, \infty].
\]
\end{Definition} 

Most of our argument relates the aperture function with information on the integrability of fractional derivatives of solutions to \eqref{eq_main}. Indeed, we consider the $(1+\alpha)$-differential quotient of the solutions to \eqref{eq_main} and control it from above by the aperture $\theta$. Conversely, we bound this $(1+\alpha)$-differential quotient from below in terms of the fractional Laplacian of the solutions. Controlling the integrability of the aperture function we transmit bounds for the fractional Laplacian in suitable Lebesgue spaces and conclude the proof. Compare with the analysis in \cite[Proposition 1.1]{cafcab}.

In the sequel, we recall the primary ingredients of the theories unlocking this connection. We start with an auxiliary lemma.

\begin{Lemma}\label{lem_reg}
For $\Omega\subset\mathbb{R}^d$, let $g:\Omega\to\mathbb{R}$ be a nonnegative and measurable function. Define $\mu_g:\mathbb{R}^+_0\to\mathbb{R}^+_0$ as
\[
\mu_g(t) := |\{x\in\Omega ~ | ~ g(x) >t\}|, \;\; t>0.
\] 
Let $\eta>0$ and $M>1$ be constants. Then for $0<p<\infty$,
\[
g\in L^p(\Omega) \Longleftrightarrow \sum_{k\geq 1}M^{pk}\mu_g(\eta M^k) = S < \infty
\]
and
\[
C^{-1}S\leq \|g\|^p_{L^p(\Omega)}\leq C(|\Omega| + S),
\]
where $C>0$ is a constant depending only on $\eta$, M and p.
\end{Lemma}

The function $\mu_g$, defined in Lemma \ref{lem_reg}, is known as the \emph{distribution function of $g$}. Next, we recall a corollary of the Calder\'on-Zygmund decomposition. See \cite[Lemma 4.2]{cafcab}. Let $Q_1$ be the unit cube and split it into $2^d$ cubes of half side. Then, split each of these $2^d$ cubes and iterate the process. The cubes obtained in this way are called dyadic cubes.

If Q is a dyadic cube different from $Q_1$, we say that $\tilde{Q}$ is the predecessor of Q if the latter is one of the $2^d$ cubes obtained by dividing $\tilde{Q}$.

\begin{Lemma}[Calder\'on-Zygmund decomposition]\label{lem_cald}
Let $A \subset B \subset Q_1$ be measurable sets and take $0<\delta<1$ such that
\begin{itemize}
\item[(a)] $|A| \leq \delta$;
\item[(b)] if Q is a dyadic cube such that $|A\cap Q| > \delta|Q|$, then $\widetilde{Q} \subset B$, where $\widetilde{Q}$ is the predecessor of Q.
\end{itemize}
Then
\[
|A| \leq \delta|B|.
\]
\end{Lemma}

Now, we can state a connection between the distribution function $\mu_\theta$ and the measure of the sets $A_M$. In fact,
\[
\mu_{\theta}(t) \leq |A_t(u,B_{1/2})|;
\]
therefore, we study the summability of
\[
\sum_{k\leq 1}M^{pk}|A_{M^k}(u,B_{1/2})|
\]
and the mechanism transmitting information from $\theta$ to $u$.

In what follows, we define the fractional Sobolev spaces. We refer to \cite[Chapter 2]{Nezza-Palatucci-Valdinoci2012} for further details. 

\begin{Definition}[Fractional Sobolev spaces]\label{def_fracsob1}
Let $s \in (0,1)$. For any $p \in [1, +\infty)$ we define $W^{s,p}(\Omega)$ as
\[
W^{s,p}(\Omega) := \left\{u \in L^p(\Omega) : \frac{|u(x)-u(y)|}{|x-y|^{\frac{d}{p} + s}} \in L^P(\Omega\times\Omega)\right\}.
\]
We equip $W^{s,p}(\Omega)$ with a norm denoted with $\|\cdot\|_{W^{s,p}(\Omega)}$ and given by
\[
\|u\|_{W^{s,p}(\Omega)} := \left(\int_{\Omega}|u|^pdx + \int_\Omega\int_\Omega \frac{|u(x) - u(y)|^p}{|x-y|^{d+sp}}{\rm d}x{\rm d}y \right)^{\frac{1}{p}}.
\]
Moreover, we denote with $[u]_{W^{s,p}(\Omega)}$ the Gagliardo seminorm of $u$, defined as
\[
[u]_{W^{s,p}(\Omega)} := \left( \int_\Omega\int_\Omega \frac{|u(x) - u(y)|^p}{|x-y|^{d+sp}}{\rm d}x{\rm d}y \right)^{\frac{1}{p}}.
\]
\end{Definition}

Sobolev spaces involving weak derivatives of higher fractional order are defined next.

\begin{Definition}[Higher fractional order Sobolev spaces]\label{def_fracsob2}
If $s > 1$, we write $s= m + \gamma$, where $m$ is an integer and $\gamma \in (0,1)$. We define
\[
W^{s,p}(\Omega) := \{u\in W^{m,p}(\Omega)\;|\;D^\alpha u \in W^{\gamma,p}(\Omega), \;\; \text{for any}\;\;    \alpha   \;\;\text{s.t.}\;\;  |\alpha|=m\}.
\]
In this case, $W^{s,p}(\Omega)$ is equipped with the following norm:
\[
\|u\|_{W^{s,p}(\Omega)} := \left(\|u\|^p_{W^{m,p}(\Omega)}+\sum_{|\alpha|=m}\|D^{\alpha}u\|^p_{W^{\gamma,p}(\Omega)}\right)^{\frac{1}{p}}.
\]
\end{Definition} 

When $p=2$, we write $W^{s,2}(\Omega) =: H^s(\Omega)$. The space $W_0^{s,p}(\Omega)$ consists of all functions $u \in W^{s,p}(\mathbb{R}^d)$ such that $u = 0$ in $\mathbb{R}^d\setminus \Omega$. In addition, $W^{-s,p}(\Omega)$ denotes the dual space of $W^{s,p}(\Omega)$.

Next, the facts we recall concern the Fourier transform and its relationship with the fractional Laplacian operator. For simplicity, we operate in the context of the Schwartz space, denoted with $\mathcal{S}$; we refer the reader to \cite[Chapter 3]{Nezza-Palatucci-Valdinoci2012}. Standard density arguments allow us to switch from $\mathcal{S}$ to $L^2(\mathbb{R}^d)$, as suitable. In what follows we recall the definition of the fractional Laplacian and the Fourier transform, to which we resort in our argument.

For $s \in (0,1)$ and $v \in {\mathcal S}$, we define the fractional Laplacian operator as
\[
(-\Delta)^sv(x) := -\frac{1}{2}C(d,s)\int_{\mathbb{R}^d}\frac{v(x+y) + v(x-y) - 2v(x)}{|y|^{d+2s}}{\rm d}y. 
\]
The Fourier transform of $u$ is defined by
\[
	{\mathcal F}v(\zeta) := \frac{1}{(2\pi)^{d/2}}\int_{\mathbb{R}^d}e^{-i\zeta\cdot x}v(x){\rm d}x. 
\]

The interaction between the Fourier transform and the fractional Laplacian operator is the subject of the upcoming proposition.

\begin{Proposition}\label{prop_fourfrac}
Let $s \in (0,1)$ and let $(-\Delta)^s\colon  \mathcal{S} \to L^2(\mathbb{R}^d)$ be the fractional Laplacian operator. Then, for any $u \in {\mathcal S}$,
\[
(-\Delta)^su = {\mathcal F}^{-1}(|x|^{2s}({\mathcal F}u))    \hspace{.2in }\text{for all}\hspace{.2in}   x \in \mathbb{R}^d.
\]
\end{Proposition}

For a proof of this fact, we refer the reader to \cite[Proposition 3.3]{Nezza-Palatucci-Valdinoci2012}. We continue with the definition of a related functional space.

\begin{Definition}
For $s\in\mathbb{R}$ we define
\[
\bar{H}^s(\mathbb{R}^d) = \left\{u \in L^2(\mathbb{R}^d) \suchthat \int_{\mathbb{R}^d}(1 + |x|^{2s})|{\mathcal F}u(x)|^2{\rm d}x < \infty \right\}.
\]
\end{Definition}

\begin{Proposition}\label{prop_coincspac}
Let $s \in (0,1)$. The fractional Sobolev space $H^s(\mathbb{R}^d)$ coincides with $\bar{H}^s(\mathbb{R}^d)$. In particular, for any $u \in H^s(\mathbb{R}^d)$
\[
[u]^2_{H^s(\mathbb{R}^d)} = 2C(n,s)^{-1}\int_{\mathbb{R}^d}|x|^{2s}|{\mathcal F}u(x)|^2{\rm d}x.
\]
\end{Proposition}

For the proof of Proposition \ref{prop_coincspac}, we mention \cite[Proposition 3.4]{Nezza-Palatucci-Valdinoci2012}. The next result concerns \emph{local} regularity for the fractional Laplacian equation of order $s$. See \cite[Theorem 1.4]{Biccari-Warma-Zuazua2017} for details.

\begin{Proposition}\label{prop_lapfrac}
Let $u \in W_0^{s,2}(\bar{\Omega})$ be the unique weak solution to
\begin{equation}\label{zuzua_result}
	\begin{cases}
		(-\Delta)^su =f&  \hspace{.2in}  \text{in} \hspace{.2in}   \Omega \\
		u = 0&   \hspace{.2in}  \text{in} \hspace{.2in}  \mathbb{R}^d\setminus \Omega.
	\end{cases}
\end{equation}
where $\Omega \subset \mathbb{R}^d$ is an arbitrary bounded open set and $s \in (0,1)$. 
If $f\in L^{p}(\Omega)$ with $1<p<\infty$, then $u\in W^{2s,p}_{\rm loc}(\Omega)$. In addition, for every $\Omega_0\Subset\Omega$ there exists $C>0$ such that 
\[
	\left\|u\right\|_{W^{2s,p}(\Omega_0)} \leq C\left(\left\|u\right\|_{L^\infty(\Omega)}+\left\|f\right\|_{L^p(\Omega)}\right),
\]
where $C=C(d,p,{\rm diam}(\Omega),{\rm dist}(\Omega_0,\partial\Omega))$.
\end{Proposition}

\section{Geometric-Measure Tools}\label{sec_tools}

This section discusses an $L^\delta$-estimate for the aperture function, introduced in Definition \ref{def_aperture}. This first level of integrability stems from the uniform ellipticity of $F$ and the integrability of $f$; see A\ref{assump_ellip} and A\ref{assump_source}. We refine such estimate further in the argument, where geometric techniques build upon the regularity available for $F=0$ in $ C ^{1,\beta}$-spaces. To be more precise, following proposition holds true:

\begin{Proposition}[$L^\delta$-estimate for the aperture function]\label{prop_delta}
Let $u \in   C (B_1)$ be a normalized viscosity solution to \eqref{eq_main}. Suppose A\ref{assump_ellip} and A\ref{assump_source} are in force. Then the $C^{1,\alpha}$-aperture function, $\theta_{1+\alpha} = \theta$, defined in Definition \ref{def_aperture}, is in $L^\delta(B_1)$, for some universal $0<\delta\ll1$. In additional, there exists a universal constant $C>0$ such that
\[
	\int_{B_{1/2}}|\theta(x)|^\delta{\rm d}x  \leq  C.
\]
Neither $0<\delta\ll 1$ nor $C>0$ depend on the parameter $\alpha$.
\end{Proposition}

Proposition \ref{prop_delta} is analogous to Lin's integral estimates, introduced in the linear setting in \cite{Lin_1986}. See \cite{caf,cafcab} for its fully nonlinear counterpart; for a more recent account of this result, see \cite{ArmSilSma_2012, Mooney}. The proof of Proposition \ref{prop_delta} is standard and follows along the same lines as in \cite[Chapter 7]{cafcab}. The important modification, from the technical viewpoint, concerns the use of $C^{1,\alpha}$-cones instead of paraboloids, as in \cite{cafcab}. Though we omit the proof, we recall some of its ingredients for further reference.

\begin{Lemma}\label{lem_7.5}
Let $u \in   C (B_{6\sqrt{d}})$ be a viscosity solution to \eqref{eq_main} in $B_{6\sqrt{d}}$. Suppose A\ref{assump_ellip} and A\ref{assump_source} are in force. Suppose further that $\Omega$ is a bounded domain such that $B_{6\sqrt{d}} \subset \Omega$. Then
\[
	\left|\underline{G}_M(u,\Omega)\cap Q_1\right|  \geq  1 - \sigma,
\] 
where $0<\sigma<1$ and $M>1$ are universal constants.
\end{Lemma}

For the proof of Lemma \ref{lem_7.5} we refer the reader to \cite[Lemma 7.5]{cafcab}.

We close this section with a decay rate for the measure of the sets $A_t(u,\Omega)\cap Q_1$. The conclusion of the next lemma is equivalent to the statement of Proposition \ref{prop_delta}. We state it next for convenience, as we resort to its formulation as a characterization of $\delta$-integrability for $\theta$.

\begin{Lemma}\label{lem_7.8}
Let $u \in   C (B_{6\sqrt{d}})$ be a viscosity solution to \eqref{eq_main} in $B_{6\sqrt{d}}$. Suppose A\ref{assump_ellip} and A\ref{assump_source} are in force and $\Omega$ is a bounded domain such that $B_{6\sqrt{d}} \subset \Omega$.  Extend $f$ by zero outside $B_{6\sqrt{d}}$. Then
\begin{equation}\label{eq_05}
|A_t(u,\Omega)\cap Q_1| \leq c_2t^{-\mu}, \;\;\forall\;t>0,
\end{equation}
where $c_2$ and $\mu$ are positive universal constants.
\end{Lemma}

The information in Proposition \ref{prop_delta}, or Lemma \ref{lem_7.8}, is instrumental in what follows. By framing an auxiliary function with $C^{1,\alpha}$-estimates in the context of those results, in the Proposition \ref{prop_delta2} we manage to improve the decay rate put forward in \eqref{eq_05}.

The next section produces an approximation lemma relating the solutions to \eqref{eq_main} with viscosity solutions to $F=0$. The regularity theory available for the latter refines the decay rate for $|A_{M^k}(u, B_{6\sqrt{d}})\cap Q_1|$, ultimately improving the integrability of the aperture function $\theta$.

\section{Higher integrability of the aperture function}\label{sec_prov}

In what follows, we refine the decay rate of the measure of specific sets, leading to improved integrability of the aperture function. The core of this section is the following proposition.

\begin{Proposition}[Improved integrability of $\theta$]\label{prop_delta2}
Let $u \in   C (B_1)$ be a normalized viscosity solution to \eqref{eq_main}. Suppose A\ref{assump_ellip}-A\ref{assump_source} are in force. {{Suppose further that $\left\|f\right\|_{L^p(B_1)}<\epsilon$, for some $\epsilon>0$ to be determined}}. Then $\theta\in L^{p}(B_1)$ and there exists a universal constant $C>0$ such that
\[
	\left\|\theta\right\|_{L^{p}(B_{1/2})}  \leq  C.
\]
\end{Proposition}

In establishing Proposition \ref{prop_delta2}, the key ingredient is an approximation lemma importing information from the solutions to $F=0$.

\begin{Lemma}[Approximation Lemma]\label{lem_7.9}
 Let $u \in   C (B_{8\sqrt{d}})$ be a normalized viscosity solution to \eqref{eq_main} in $B_{8\sqrt{d}}$. Suppose that A\ref{assump_ellip}-A\ref{assump_source} are in force. Given $\delta >0$, there exists $0< \epsilon < \delta$ such that, if $\|f\|_{L^p(B_{8\sqrt{d}})} \leq \epsilon$, one can find a function $h \in  C _{{\rm loc}}^{1,\alpha_0}(B_{7\sqrt{d}})$ satisfying
\[
\|u-h\|_{L^\infty_{{\rm loc}}\left(\overline{B}_{7\sqrt{d}}\right)}  \leq  \delta,
\]
where $\alpha_0 =\alpha_0(d,\lambda, \Lambda)\in(0,1)$ accounts for the regularity available for $F=0$. In addition, there exists a universal constant $C=C(d, \lambda, \Lambda, \alpha_0)$ such that
\[
	\left\|h\right\|_{ C ^{1,\alpha_0}(B_{13\sqrt{d}/2})} \leq C\left\|h\right\|_{L^\infty(B_{20\sqrt{d}/3})}.
\]
\end{Lemma}
\begin{proof}
We argue  as in \cite{M5}, by supposing, seeking a contradiction, that the statement of the proposition is false. In this case, there exists a number $\delta_0>0$, a sequence of $(\lambda,\Lambda)$-elliptic operators $(F_n)_{n\in\mathbb{N}}$ and sequences of functions $(u_n)_{n\in\mathbb{N}}$ and $(f_n)_{n\in\mathbb{N}}$, satisfying
\begin{equation}\label{eq_08}
	F_n(D^2u_n) = f_n \;\;\text{ in } \;\; B_{8\sqrt{d}}
\end{equation}
and
\[
\|f_n\|_{L^p(B_{8\sqrt{d}})} \leq \frac{1}{n},
\]
with
\begin{equation}\label{eq_09}
\|u_n - h\|_{L^{\infty}(B_{6\sqrt{d}})} > \delta_0,
\end{equation}
for all $h \in  C ^{1,\alpha_0}_{{\rm loc}}(\overline{B}_{7\sqrt{d}})$.

From the regularity available for \eqref{eq_08}, the family $(u_n)_{n\in\mathbb{N}}$ is equibounded in some H\"older space $ C ^\beta$, for some $\beta\in(0,1)$ unknown, though universal. Therefore, there exists a convergent subsequence, still denoted with $(u_n)_{n\in\mathbb{N}}$, and $u_\infty \in  C ^{\frac{\beta}{2}}_{\text{loc}}(B_{7\sqrt{d}})$, such that $u_n\to u_\infty$ locally uniformly.

In addition, by already classical arguments, there exists a $(\lambda,\Lambda)$-elliptic operator $F_\infty$ such that $F_n\to F_\infty$. At this point, standard stability results \cite[Theorem 3.8]{CCKS} imply that $u_\infty$ solves
\[
F_\infty(D^2u_\infty) = 0 \;\;\text{ in }\;\; B_{7\sqrt{d}}.
\]
Since $F$ is a $(\lambda,\Lambda)$-uniformly elliptic operator, we have  $u_\infty  \in  C ^{1,\alpha_0}_{{\rm loc}}(\overline{B}_{7\sqrt{d}})$. By taking $h\equiv u_\infty$, we get a contradiction with \eqref{eq_09}, and the proof is complete.
\end{proof}

In the remainder of this section, we suppose the source term satisfies the smallness regime prescribed in Lemma \ref{lem_7.9}. That is,
\[
	\left\|f\right\|_{L^p(B_{8\sqrt{d}})} \leq \epsilon,
\]
for $0<\epsilon\ll1$ as in the statement of Lemma \ref{lem_7.9}, though yet to be determined.

\begin{Lemma}\label{lem_7.10}
Let $u \in   C (B_{8\sqrt{d}})$ be a normalized viscosity solution to \eqref{eq_main} in $B_{8\sqrt{d}}$. Suppose that A\ref{assump_ellip}-A\ref{assump_source} are in force. For every $\rho\in(0,1)$ one can choose $\varepsilon>0$ such that, if $\left\|f\right\|_{L^d(\Omega)}<\epsilon$ and
\[
	-|x|^{1+\alpha} \leq u(x) \leq |x|^{1+\alpha}\hspace{.4in} \mbox{in}\hspace{.4in} \Omega\backslash B_{6\sqrt{d}},
\]
for fixed $\alpha\in(0,1)$, then 
\begin{equation}\label{eq_10}
|G_M(u,\Omega)\cap Q_1| \geq 1-\rho,
\end{equation}
where $M>1$ in a universal constant.
\end{Lemma}
\begin{proof}
Fix $0< \delta < 1$, yet to be determined. Let $h$ be the approximating function whose existence follows from Lemma \ref{lem_7.9}. 
Since
\[
\|u-h\|_{L^\infty(\Omega)} \leq 1,
\]
it follows that
\[
-1-|x|^{1+\alpha} \leq h(x) \leq 1+|x|^{1+\alpha} \hspace{.4in}\text{ in }\hspace{.4in}\Omega\backslash B_{6\sqrt{d}}.
\]
Therefore, there exists $1<N=N(d,C)$ such that
\begin{equation}\label{eq_11}
Q_1 \subset G_N(h, \Omega).
\end{equation}
Define
\[
w(x)=\frac{(u-h)(x)}{\delta}.
\]
Notice that $w$ solves
\[
	\tilde{F}(D^2v,x) - \tilde{f}(x) = 0,
\]
where
\[
\tilde{F}(M,x) := \frac{1}{\delta}F\left(\delta M + D^2h\right), 
\]
and
\[
\tilde{f}(x) := \frac{1}{\delta}f(x).
\]
As a consequence of the former inequality, we have
\[
	\|\tilde{f}\|_{L^d} \leq \frac{1}{\delta}\|f\|_{L^d} \leq \epsilon. 
\]
Hence, $w$ is entitled to the conclusions of Lemma \ref{lem_7.5} in $\Omega$. Because of Lemma \ref{lem_7.8}, we obtain
\[
|A_t(w,\Omega)\cap Q_1|\leq t^{-\mu}, \;\;\text{ for all }\;\;t>0.
\]
It follows that
\[
|A_s(u-h, \Omega)\cap Q_1| \leq cs^{-\mu}\delta^\mu\;\;\text{ for all }\;\;s>0.
\]
By choosing $\delta$ small enough we get
\[
|G_N(u-h, \Omega)\cap Q_1| \geq 1-\delta^\mu \geq 1-\rho.
\]
\end{proof}

The proof of Lemma \ref{lem_7.10} sets the proximity-regime encoded by $\delta>0$. As a by-product, it sets the smallness condition on the $L^p$-norm of the source term $f$, encoded by $\epsilon>0$ in the statement of Lemma \ref{lem_7.9}. In the remainder of this section, these constants remain fixed.

\begin{Lemma}\label{lem_7.11}
Let $u \in   C (B_{8\sqrt{d}})$ be a normalized viscosity solution to \eqref{eq_main} in $B_{8\sqrt{d}}$. Suppose that A\ref{assump_ellip}-A\ref{assump_source} are in force, { {with $\left\|f\right\|_{L^d(\Omega)}<\epsilon$}}. If
\begin{equation}\label{eq_012}
G_1(u, \Omega)\cap Q_3\not= \emptyset,
\end{equation}
then
\[
|G_M(u, \Omega)\cap Q_1| \geq 1-\rho,
\]
with $M$ and $\rho$ as in Lemma \ref{lem_7.10}.
\end{Lemma}
\begin{proof}
Let $x_1 \in G_1(u,\Omega)\cap Q_3$. Hence, there exists an affine function $L(x)$, such that
\[
-\frac{1}{2}|x-x_1|^{1+\alpha} \leq u(x)-L(x) \leq \frac{1}{2}|x-x_1|^{1+\alpha} \;\; \text{ in }\;\; \Omega.
\]
Define 
\[
v(x) = \frac{u(x)-L(x)}{c(d)},
\]
where $c(d)$ is a constant depending only on $d$, large enough as to guarantee $|v(x)| \leq 1$ and
\[
|v(x)| \leq |x|^{1+\alpha} \;\;\text{ in }\;\;\Omega\backslash B_{6\sqrt{d}}.
\]
In addition, $v$ solves	
\[
	\tilde{F}(D^2u)-\tilde{f}(x) = 0,
\]
where
\[
\tilde{F}(M) := \frac{1}{c(d)}F(c(d)M),
\]
and
\[
\tilde{f}(x) := \frac{1}{c(d)}f(x).
\]
Lemma \ref{lem_7.10} yields
\[
|G_M(v,\Omega)\cap Q_1| \geq 1-\rho,
\]
therefore
\[
	|G_{c(d)M}(u,\Omega)\cap Q_1|\geq 1-\rho
\]
and the Lemma is proven.
\end{proof}

\bigskip

The following result resorts once again to the Calder\'on-Zygmund decomposition.

\begin{Lemma}\label{lem_7.12}
Let $u \in   C (B_{8\sqrt{d}})$ be a normalized viscosity solution to \eqref{eq_main} in $B_{8\sqrt{d}}$. Suppose that A\ref{assump_ellip}-A\ref{assump_source} are in force, { {with $\left\|f\right\|_{L^d(\Omega)}<\epsilon$.}} Extend f by zero outside $B_{8\sqrt{d}}$ and set
\[
A:=A_{M^{k+1}}(u, B_{8\sqrt{d}})\cap Q_1,
\]
\[
B:=\left\lbrace A_{M^{k}}(u, B_{8\sqrt{d}})\cap Q_1\right\rbrace\cup \left\lbrace x\in Q_1:m(f^d)(x) \geq c_3^dM^{kd}\right\rbrace,
\]
for $k\in\mathbb{N}$. Then
\[
|A| \leq \rho|B|,
\]
where $M>1$ is a universal constant and $c_3>0$ depends only on d, $\lambda$, $\Lambda$ and $\rho$.
\end{Lemma}
\begin{proof}
We start by noticing that $|u|\leq 1 \leq |x|^{1 + \alpha}$ in $B_{8\sqrt{d}}\backslash B_{6\sqrt{d}}$. Hence  Lemma \ref{lem_7.10} applied with $\Omega = B_{8\sqrt{d}}$, implies 
\[
|G_{M^{k+1}}(u,B_{8\sqrt{d}})\cap Q_1| \geq |G_{M^{k}}(u,B_{8\sqrt{d}})\cap Q_1| \geq 1-\rho.
\] 
It leads to $|A| \leq \rho$.

The remainder of the proof relies on the Calder\'on-Zygmund decomposition, as stated in Lemma \ref{lem_cald}. Hence, we need to show that if $Q = Q_{1/2^i}(x_0)$ is a dyadic cube $Q_1$ such that
\begin{equation}\label{eq_013}
|A_{M^{k+1}}(u,B_{8\sqrt{d}})\cap Q| = |A\cap Q| > \rho|Q|,
\end{equation}
we have $\tilde{Q} \subset B$. We suppose otherwise and produce a contradiction. Suppose that $\tilde{Q} \not\subset B$ and let $x_1$ be such that
\begin{equation}\label{eq_014}
x_1\in\tilde{Q}\cap G_{M^k}(u,B_{8\sqrt{d}})
\end{equation}
and
\begin{equation}\label{eq_015}
m(f^d)(x_1) \leq \left(c_3M^{k}\right)^d.
\end{equation}
Consider as before the transformation 
\begin{equation}\label{eq_transf2}
x = x_0+\frac{1}{2^i}y, \;\;x\in B_{8\sqrt{d}},
\end{equation}
and define
\[
v(y) = \frac{2^{2i}}{M^k}u\left(x_0 + \frac{1}{2^i}y\right).
\]
Finally, let $\tilde{\Omega}$ be the image of $B_{8\sqrt{d}}$ under the transformation \eqref{eq_transf2}. We need to verify that $v$ satisfies the hypothesis of Lemma \ref{lem_7.11}. Note that $v$ solves
\[
	\tilde{F}(D^2u) - \tilde{f}(x) = 0 \;\; in \;\; \tilde{\Omega},
\]
where
\[
\tilde{F}(N) := \frac{1}{M^k}F\left(M^kN\right)
\]
and
\[
\tilde{f}(x) := \frac{1}{M^{k}}f\left(x_0+\frac{1}{2^i}x\right).
\]
Since $B_{8\sqrt{d}} \subset \tilde{\Omega}$, the function $v$ satisfies the equation in $B_{8\sqrt{d}}$, in the viscosity sense. Furthermore $|x_1-x_0|_{\infty} \leq 3/2^{i+1}$ implies that $B_{8\sqrt{d}/2^i}(x_0) \subset Q_{19\sqrt{d}/2^i}(x_1)$. Hence 
\[
\begin{array}{rcl}
\|\tilde{f}\|^d_{B_{8\sqrt{d}}} & = &\displaystyle \frac{2^{id}}{M^{kd}}\int_{B_{8\sqrt{d}/2^i}(x_0)}|f(x)|^d{\rm d}x \vspace{0.2cm}\\
 &\leq & \displaystyle \frac{c(d)}{M^{kd}}\frac{1}{|Q_{19\sqrt{d}/2^i}|}\int_{Q_{19\sqrt{d}/2^i}(x_1)}|f(x)|^d{\rm d}x \\
 & \leq & c(d)c_3^d \\
 & \leq & \epsilon,
\end{array}
\]
for $c_3$ small enough.

Now, by \eqref{eq_014} there exists a convex and a concave $ C ^{1,\alpha}$-cones of opening $M^k$, $\psi_1$ and $\psi_2$ respectively, such that $\psi_1$ touches $u$ from above at $x_1$ and $\psi_2$ touches $u$ from below at $x_1$. Define
\[
\tilde{\psi_1}(y) := \psi_1\left(x_0+\frac{1}{2^i}y\right)
\]
and 
\[
\tilde{\psi_2}(y) := \psi_2\left(x_0+\frac{1}{2^i}y\right).
\]
It is easy to see that $\tilde{\psi_1}$ (resp. $\tilde{\psi_2}$) touches $v$ from above (respectively from below) in a point $y_1$ such that $x_1 = x_0 + \frac{1}{2^i}y_1$. Therefore $G_1(v, \tilde{\Omega}) \not= \emptyset$. By Lemma \ref{lem_7.11} we obtain
\[
|G_M(v, \tilde{\Omega})\cap Q_1| \geq 1-\rho = (1-\rho)|Q_1|.
\] 
Hence
\[
|G_{M^{k+1}}(u, B_{8\sqrt{d}})\cap Q| \geq (1-\rho)|Q|,
\]
which implies
\[
|A_{M^{k+1}}(u,B_{8\sqrt{d}})\cap Q| \leq \rho|Q|.
\]
This fact produces a contradiction with \eqref{eq_013} and finishes the proof.
\end{proof}

Next, we detail the proof of Proposition \ref{prop_delta2}.

\begin{proof}[Proof of Proposition \ref{prop_delta2}]
Let $M$ be as in Lemma \ref{lem_7.12} and take $\rho$ such that
\[
\rho M^{p} = \frac{1}{2}.
\]
For $k \geq 0$, define
\[
\alpha_k :=|A_{M^k}(u,B_{8\sqrt{d}})\cap Q_1| 
\]
and
\[
\beta_k:= \Big|\left\lbrace x\in Q_1~ | ~  m(f^d)(x)\geq \left(c_3M^{k}\right)^d\right\rbrace\Big|.
\]
By Lemma \ref{lem_7.12} we obtain $\alpha_{k+1} \leq \rho(\alpha_k + \beta_k)$. Hence
\begin{equation}\label{eq_016}
\alpha_k \leq \rho^k + \sum_{i=0}^{k-1}\rho^{k-i}\beta_i.
\end{equation}
Since $f^d \in L^{p/d}(B_{8\sqrt{d}})$, we have that $m(f^d) \in L^{p/d}(B_{8\sqrt{d}})$ and 
\[
\|m(f^d)\|_{L^{p/d}(B_{8\sqrt{d}})} \leq c\|f\|^d_{L^p(B_{8\sqrt{d}})} \leq C.
\]
Therefore, by Lemma \ref{lem_reg} we obtain
\[
\sum_{k\geq 0}\left(M^{d}\right)^\frac{pk}{d}|x\in Q_1~ | ~ m(f^d)(x)\geq c_3^dM^{dk}| \leq C.
\]
The former inequality implies
\begin{equation}\label{eq_017}
\sum_{k\geq 0}M^{pk}\beta_k \leq C.
\end{equation}
Since $B_{1/2} \subset Q_1$, the distribution function of $\theta$ is bounded from above as follows:
\[
\mu_\theta(t) \leq |A_t(u,B_{1/2})| \leq |A_t(u,B_{8\sqrt{d}})\cap Q_1|.
\]
Hence
\[
\begin{array}{rcl}
\displaystyle \sum_{k\geq 1}M^{pk}\alpha_k & \leq & \displaystyle \sum_{k\geq 1}\left(\rho M^{p}\right)^k + \sum_{k\geq 1}\sum_{i=0}^{k-1}\rho^{k-i}M^{pk}\beta_i \vspace{0.1cm} \\
 & = & \displaystyle \sum_{k\geq 1}2^{-k} + \sum_{k\geq 1}\sum_{i=0}^{k-1}\rho^{k-i}M^{p(k-i)}M^{pi}\beta_i \vspace{0.1cm} \\
 & = & \displaystyle \sum_{k\geq 1}2^{-k} + \sum_{k\geq 1}\sum_{i=0}^{k-1}2^{-(k-i)}M^{pi}\beta_i \vspace{0.1cm} \\  
 & = & \displaystyle \sum_{k\geq 1}2^{-k} + \left(\sum_{i\geq 0}M^{pi}\beta_i \right)\left(\sum_{j\geq 1}2^{-j}\right) \vspace{0.1cm} \\
 & \leq & C.
\end{array}
\]
Applying Lemma \ref{lem_reg} once again we conclude that $\|\theta\|_{L^{p}(B_{1/2})} \leq C$ and the proof is complete.
\end{proof}

\section{Fractional diffusion and a proof of Theorem \ref{theo_main} }\label{sec_proofs}

In this section, we detail the proof of Theorem \ref{theo_main}. 
\begin{proof}[Proof of Theorem \ref{theo_main}]
We start by noticing that if the graph of $u$ is touched at a point $x_0$ by a $C^{1,\alpha}$-cone, we may assume $x_0$ is the vertex of the cone.
Let $\psi$ be a $ C ^{1,\alpha}$-cone of opening $\pm M$ and vertex $x_0$. We have:
\[
\begin{array}{lcl}
\Delta^{1+\alpha}_h\psi(x_0) & := & \frac{\psi(x_0+h) + \psi(x_0-h) - 2\psi(x_0)}{|h|^{1+\alpha}} \vspace{0.2cm}\\
 &= &  \pm M.
\end{array}
\]

Also, notice that touching $u$ strictly in $B_{\frac{1}{10}}(x_0)$  from above at $x_0$ by a convex $ C ^{1,\alpha}$-cone $\psi$ of opening $M$ and vertex $x_0$ gives for all $0< h < \frac{1}{10}$  
\[
	\begin{array}{lcl}
	\Delta^{1+\alpha}_h u(x_0) & := & \displaystyle \frac{u(x_0+h) + u(x_0-h) - 2u(x_0)}{|h|^{1+\alpha}} \vspace{0.2cm}\\
 						&< &  \displaystyle  \frac{\psi(x_0+h) + \psi(x_0-h) - 2\psi(x_0)}{|h|^{1+\alpha}} \vspace{0.2cm}\\
						&\le& \displaystyle  \theta(u,B_{1/2})(x_0).	
	\end{array}
\]
Similarly, touching $u$ strictly in $B_{\frac{1}{10}}(x_0)$  from below at $x_0$ by a concave $ C ^{1,\alpha}$-cone $\psi$ of opening $M$ and vertex $x_0$ gives, for all $0< h < \frac{1}{10}$:
$$
	-\theta(u,B_{1/2})(x_0)	 <  \Delta^{1+\alpha}_h u(x_0).
$$
Hence, by hypothesis,
$$
	\| \Delta^{1+\alpha}_h u  \|_{L^{p}(B_{1/2})} \leq C,
$$
uniformly for all $0< h < \frac{1}{10}$, for $0<\alpha < \alpha_0$. At this point, we set $\varphi := u\chi_{B_{1/2}}$ in $\mathbb{R}^d$. Then we have that $\varphi \in L^{\infty}(\mathbb{R}^d)$. Next, for 
\[
1< \sigma < 1 + \alpha,
\]
we define the singular integral operator:
\[
I_{\sigma/2}(v)(x_0) := \int_{\mathbb{R}^d}\frac{v(x_0+y) + v(x_0-y) - 2v(y)}{|y|^{d+\sigma}}{\rm d}y.
\]
Notice that, up to constants, $I_{\sigma/2}(v) \sim \Delta^{\sigma/2}(v)$. For $x_0 \in B_{1/2}$ we estimate
\[
\begin{array}{lcl}
I_{\sigma/2}(\varphi)(x_0) & = & \displaystyle \int_{\mathbb{R}^d}\frac{\varphi(x_0+y) + \varphi(x_0-y) - 2\varphi(y)}{|y|^{d+\sigma}} {\rm d}y \vspace{0.2cm} \\
 & = & \displaystyle \int_{B_{1/10}}\frac{u(x_0+y) + u(x_0-y) - 2u(y)}{|y|^{d+\sigma}} {\rm d}y \vspace{0.2cm} \\
 &   & + \displaystyle \int_{B_1\setminus B_{1/10}}\frac{u(x_0+y) + u(x_0-y) - 2u(y)}{|y|^{d+\sigma}} {\rm d}y \vspace{0.2cm} \\
 & \leq & \displaystyle \theta(u, B_{1/2})(x_0)\int_{B_{1/10}}\frac{dy}{|y|^{d-\mu}} + C\|u\|_{L^{\infty}(B_{1})} \vspace{0.2cm} \\
 & = & \displaystyle \frac{C}{\mu 10^d}\cdot \left ( \theta(u,B_{1/2}) + \|u\|_{L^{\infty}(B_{1})} \right )
\end{array}
\]
where $\mu := 1 + \alpha - \sigma,$ and $C$ is a universal constant. Hence, we have proven that
\[
	(-\Delta)^{\sigma/2}\varphi \in L^{p}(B_{1/2}).
\]
By setting $g:= (-\Delta)^{\sigma/2}\varphi$ in $B_{1/2}$ we conclude that $\varphi$ satisfies
\[
	\begin{cases}
		(-\Delta)^{\sigma/2}\varphi = g&\hspace{.2in} \mbox{in}\hspace{.2in}   B_{1/2} \\
		\varphi = 0 &\hspace{.2in} \mbox{in}\hspace{.2in}     \mathbb{R}^d\setminus B_{1/2} .
	\end{cases}
\]
Now, we aim at showing that $\varphi \in W^{\sigma/2, 2}(\mathbb{R}^d)$. Extend $g$ by zero outside $B_{1/2}$. It is clear that $g \in L^2(\mathbb{R}^d)$. Hence ${\mathcal F}(g) \in L^2(\mathbb{R}^d)$. In addition, since $\varphi \in L^2(\mathbb{R}^d)$, ${\mathcal F}(\varphi) \in L^2({\mathbb{R}^d})$ as well. By Proposition \ref{prop_fourfrac} (applied to functions in $L^2(\mathbb{R}^d)$)
\[
{\mathcal F}(g)(x) = (|x|)^{\sigma}{\mathcal F}(\varphi).
\]

Furthermore, we have that
\[
(1 + |x|^{\sigma}){\mathcal F}(\varphi) = {\mathcal F}(\varphi) + {\mathcal F}(g).
\]
It follows that $(1 + |x|^\sigma){\mathcal F}(\varphi) \in L^{2}(\mathbb{R}^d)$. In particular, 
\[
\int_{\mathbb{R}^d}(1 + |x|^{\sigma})|{\mathcal F}(\varphi)(x)|^2dx < \infty.
\]
Hence, by Proposition \ref{prop_coincspac}, we conclude that $\varphi \in W^{\sigma/2, 2}(\mathbb{R}^d)$. Finally, by Proposition \ref{prop_lapfrac}, we obtain that $\varphi \in W^{\sigma,p}_{\rm loc}(B_{1/2})$. Therefore
\[
u \in W^{\sigma,p}_{\rm loc}(B_{1/2}),
\]
with universal estimates, which ends the proof.
\end{proof}

\begin{Remark}[Escauriaza's exponent] 
We believe it is possible to extend our results to the range $p_0< p\leq d$, where $p_0$ is the Escauriaza's exponent. Indeed, it suffices to replace Lemma \ref{lem_7.8}  and Proposition \ref{lem_7.9} in the present paper with Lemmas 4 and 5 in \cite{escauriaza}. 
\end{Remark}

\bigskip

\noindent{\bf Acknowledgements:} EP is partly funded by FAPERJ (E-26/200.002/2018), CNPq-Brazil (433623/2018-7, 307500/2017-9), and Instituto Serrapilheira (1811-25904).
MS is funded by PUC-Rio Archimedes Fund. ET thanks UCF start-up fundings. This work was partially supported by the Centre for Mathematics of the University of Coimbra - UIDB/00324/2020, funded by the Portuguese Government through FCT/MCTES.

\vspace{.3in}

\noindent\textsc{Edgard A. Pimentel}\\
University of Coimbra\\
CMUC, Department of Mathematics,\\
3001-501 Coimbra, Portugal\\
and \\
Department of Mathematics\\
Pontifical Catholic University of Rio de Janeiro -- PUC-Rio\\
22451-900, G\'avea, Rio de Janeiro-RJ, Brazil\\
\noindent\texttt{edgard.pimentel@mat.uc.pt}

\bigskip

\noindent\textsc{Makson S. Santos}\\
Instituto Superior T\'ecnico\\Department of Mathematics\\
1049-001, Lisbon, Portugal\\
\noindent\texttt{makson.santos@tecnico.ulisboa.pt}

\bigskip

\noindent\textsc{Eduardo V. Teixeira (Corresponding Author)}\\
University of Central Florida\\
4393 Andromeda Loop N, Orlando, FL 32816, USA\\
\noindent\texttt{eduardo.teixeira@ucf.edu}

\end{document}